\documentclass[a4paper]{article}
\usepackage{latexsym}
\usepackage{amsmath,amssymb,amsfonts,graphicx}

\newtheorem{prethm}{{\bf Theorem}}

\newenvironment{thm}{\begin{prethm}{\hspace{-0.5
               em}{\bf.}}}{\end{prethm}}

\newtheorem{prepro}[prethm]{Proposition}

\newtheorem{prelem}[prethm]{Lemma}
\newenvironment{lem}{\begin{prelem}{\hspace{-0.5
               em}{\bf.}}}{\end{prelem}}

\newtheorem{precor}[prethm]{Corollary}
\newenvironment{cor}{\begin{precor}{\hspace{-0.5
               em}{\bf.}}}{\end{precor}}

\newtheorem{prerem}[prethm]{{\bf Remark}}
\newenvironment{rem}{\begin{prerem}\em{\hspace{-0.5
              em}{\bf.}}}{\end{prerem}}

\newtheorem{preexample}{{\bf Example}}

\newtheorem{preproof}{{\bf Proof.}}

\newenvironment{proof}[1]{\begin{preproof}{\rm
               #1}\hfill{$\Box$}}{\end{preproof}}

\newcommand{\1}{{\bf 1}}
\newcommand{\p}{\top}

\newcommand{\la}{\lambda}
\newcommand{\li}{{\cal L}}
\newcommand{\rk}{{\rm rank}}
\newcommand{\A}{{\cal A}}
\renewcommand{\thefootnote}

\title{Spanning trees and even integer eigenvalues of graphs}
\author{\sc Ebrahim Ghorbani %\thanks{The research of the author was in part supported by a grant from IPM (No.~91050114).}
\\
{\small { Department of Mathematics,
K.N. Toosi University of Technology,}}\\
 {\small { P.O. Box 16315-1618, Tehran, Iran}}\\
 {\small { School of Mathematics, Institute
for Research in Fundamental Sciences (IPM),}} \\
 {\small { P.O. Box 19395-5746, Tehran, Iran}}\\
 {\tt\small e\_ghorbani@ipm.ir}}
\date{}
\begin{document}
\maketitle

\begin{abstract}
For a graph $G$, let $L(G)$ and $Q(G)$ be the Laplacian and signless Laplacian matrices of $G$, respectively, and $\tau(G)$ be the number of spanning trees of $G$.
We prove that if $G$ has an odd number of vertices and $\tau(G)$ is not divisible by $4$, then  (i) $L(G)$ has no even integer eigenvalue,
 (ii) $Q(G)$ has no integer eigenvalue $\la\equiv2\pmod4$, and (iii) $Q(G)$ has at most one eigenvalue $\la\equiv0\pmod4$ and such an  eigenvalue is  simple.
 As a consequence, we extend previous results by Gutman and Sciriha and by Bapat on the nullity of adjacency matrices of the line graphs.
 We also show that if $\tau(G)=2^ts$  with $s$ odd, then the multiplicity of any even integer eigenvalue of $Q(G)$ is at most $t+1$.
 Among other things, we prove that if $L(G)$ or $Q(G)$ has an even integer eigenvalue of multiplicity at least $2$, then $\tau(G)$ is divisible by $4$.
As a very special case of this result, a conjecture by Zhou et al. [On the nullity of connected graphs with least eigenvalue at least $-2$, Appl. Anal. Discrete Math. 7 (2013), 250--261]
on the nullity of adjacency matrices of the line graphs of unicyclic graphs follows.

\vspace{3mm}
\end{abstract}

\section{Introduction}

The graphs we consider are simple, that is, without loops or multiple edges.
Let $G$ be a graph.
The {\em order} of $G$ is the number of vertices of $G$.
We denote by $A(G)$ the adjacency matrix,
 by $\li(G)$  the line graph and by $\tau(G)$ the number of spanning trees of $G$.

The purpose of this paper is to study the interconnection between $\tau(G)$ and the multiplicities of even integer eigenvalues of
$A(\li(G))$.
Our motivation comes partly from the previous works by several authors on the connection between $\tau(G)$ and the multiplicity of the zero eigenvalue, i.e. the nullity of $A(\li(G))$. A brief review of the previous results is in order.
Doob \cite{doob} proved that the binary rank (i.e. the rank over the two-element field) of $A(\li(G))$ for any connected graph $G$ of order $n$  is $n-1$ if $n$ is odd, and $n-2$
if $n$ is even. This result was stated and proved in the context of Matroid Theory. %(We give a simple proof of that in Section~3.)
Considering the rank over the real numbers,
Sciriha \cite{sci} showed that the order of every tree whose line graph is
singular is even and also that the nullity of the line graph of a tree is at
most one. A new proof of the latter result appeared later in \cite{gut}.
%Gutman and Sciriha \cite{gut} showed
% that the nullity (over reals) of $A(\li(T))$ for any tree $T$ is at most 1 and
%if $A(\li(T))$ is singular, then $T$ has an even order. Indeed, this is an immediate consequence of Doob's result.
These results can also be deduced from Doob's work.
 Recently, Bapat \cite{bapat} found an interesting generalization by proving that if $\tau(G)$ is odd, then $A(\li(G))$ has nullity at most 1. He also showed that a bipartite graph $G$ with odd $\tau(G)$ and with  singular $A(\li(G))$
must have even order.
We extend these results to the following.

\begin{thm}\label{t+1} Let $G$ be a connected graph and $\tau(G)=2^ts$  with $s$ odd. Then the multiplicity of any even integer $\la\ne-2$ as an eigenvalue of $A(\li(G))$ is at most $t+1$.
\end{thm}

\begin{thm}\label{mult2line} Suppose that $G$ is a graph with odd order and that $\tau(G)$ not divisible by $4$. If  $\la\ne-2$ is an even integer eigenvalue of $A(\li(G))$, then $\la\equiv2\pmod4$, $\la$ is a simple eigenvalue, and $A(\li(G))$ has at most one such eigenvalue.
\end{thm}

\begin{cor}  If a graph $G$ has odd order and $\tau(G)$ is not divisible by $4$, then $A(\li(G))$ is nonsingular.
\end{cor}

\begin{thm}\label{oddline} If $A(\li(G))$ has an even integer eigenvalue $\la\ne-2$ of multiplicity at least $2$, then $\tau(G)$ is divisible by $4$.
\end{thm}

Since even integer eigenvalues of $A(\li(G))$ and the signless Laplacian matrix $Q(G)$ are the same modulo a shift (see Section~2) it is enough to consider those of  $Q(G)$ as we do in what follows.
The rest of the paper is organized as follows. In Section~2, we recall some necessary preliminaries. In Section~3, we give a simple proof for Doob's result which will be used later on.
In Section~4, the proofs of Theorems~\ref{t+1}, \ref{mult2line}, and \ref{oddline} in terms of $Q(G)$ are given along with some improvements and similar results for the eigenvalues of the Laplacian matrix $L(G)$.

\section{Preliminaries}

By $X=X(G)$ we denote the $0,1$ vertex-edge incidence matrix of $G$. If we orient each edge of $G$, then $D=D(G)$ will denote
the $0,\pm1$ vertex-edge incidence matrix of the resulting graph. The  Laplacian
matrix of $G$ is $L=L(G)=DD^\p$ and the  signless Laplacian matrix of $G$ is $Q=Q(G)=XX^\p$.
Note that the Laplacian does not depend on the orientation.
The matrices $L$ and $Q$ are positive semidefinite.
The incidence matrix of $G$ and the adjacency matrix of $\li(G)$ satisfy the following \cite[p.~18]{biggs}
\begin{equation}\label{A+2I}
A(\li(G))+2I=X^\p X.
\end{equation}
Recall that for a matrix $M$, the matrices $MM^\p$ and $M^\p M$ have the same nonzero eigenvalues with the same multiplicities.
This together with (\ref{A+2I}) implies that
the matrices $A(\li(G))+2I$ and $Q(G)$ have the same nonzero eigenvalues with the same multiplicities.
In particular, the multiplicity of eigenvalue $2$ for $Q(G)$ is the same as the nullity of $A(\li(G))$ \cite{gut}.
Therefore, studying even integer eigenvalues of $A(\li(G))$ and those of $Q(G)$ are equivalent.

 We denote the vertex set and the edge set of $G$ by $V(G)$ and $E(G)$, respectively.
 If $S\subseteq E(G)$, then $\langle S\rangle$ denotes the induced subgraph on $S$.
 For a matrix $M$ with $R,S$ being subsets of row and column indices of $M$, respectively, we denote the submatrix with row indices from $R$ and column indices from $S$ by $M(R,S)$.

The following two lemmas describe the invertible submatrices of $D$ and $X$. For the first one we refer to pp. 32 and 47 of \cite{biggs} and for the second one to p. 20 of \cite{bapatBook}.

\begin{lem} \label{invertD} Let $G$ be a graph and $R\subseteq V(G)$, $S\subseteq E(G)$ with $|R|=|S|\ge1$.
Let $V_0$ denote the vertex set of $\langle S\rangle$.  Then $D(R,S)$ is invertible if and only if the
following conditions are satisfied:
\begin{itemize}
  \item[\rm(i)] $R$ is a subset of $V_0$.
  \item[\rm(ii)]$\langle S\rangle$ is a forest.
  \item[\rm(iii)]$V_0\setminus R$ contains precisely one vertex from each connected component
  of $\langle S\rangle$.
\end{itemize}
Moreover, if $D(R,S)$ is invertible, then $\det(D(R,S))=\pm1$.
\end{lem}

\begin{lem} \label{invertX} Let $G$ be a graph and $R\subseteq V(G)$, $S\subseteq E(G)$ with $|R|=|S|\ge1$.
Let $V_0$ denote the vertex set of $\langle S\rangle$. Then $X(R,S)$ is invertible if and only if the
following conditions are satisfied:
\begin{itemize}
  \item[\rm(i)] $R$ is a subset of $V_0$.
  \item[\rm(ii)] each connected component of $\langle S\rangle$ is either a tree or a unicyclic graph with odd cycle.
  \item[\rm(iii)]$V_0\setminus R$ contains precisely one vertex from each tree in $\langle S\rangle$.
\end{itemize}
Moreover, if $X(R,S)$ is invertible, then  $\det(X(R,S))=\pm2^c$ where $c$ is the number of components of $\langle S\rangle$ which are unicyclic with odd cycle.
\end{lem}

 The nullity of $L(G)$ and $Q(G)$ are respectively equal to the number of components and to the number of bipartite components of $G$.
Let
$$p_Q(x)=x^n+q_1x^{n-1}+\cdots+q_n,~~~p_L(x)=x^n+\ell_1x^{n-1}+\cdots+\ell_{n-1}x$$
be the characteristic polynomials of $Q$ and $L$, respectively.
A spanning subgraph of $G$ whose components are trees or unicyclic graphs with odd cycles is called a
{\em TU-subgraph} of $G$. Suppose that a TU-subgraph $H$ of $G$ contain $c$ unicyclic graphs and trees
$T_1, T_2,\ldots, T_s$. Then the weight $W(H)$ of $H$ is defined by $$W(H) = 4^c\prod_{i=1}^s(1 + e(T_i )),$$
where $e(T_i)$ denotes the number of edges of $T_i$.
The weight of an acyclic subgraph, that is, a union of trees, is defined similarly with $c=0$.
We shall express the coefficients of $p_Q(x)$ and $p_L(x)$ in terms of the weights of TU-subgraphs and acyclic subgraphs of $G$.

 By the Matrix-Tree Theorem, for any $1\le i,j\le n$,
$\tau(G)$ is equal to $(-1)^{i+j}$ times the determinant of the submatrix of $L(G)$ obtained by eliminating the $i$th row and $j$th column.
Consideration the trace of the adjugate of $L(G)$ yields that $\ell_{n-1}=(-1)^{n-1}n\tau(G)$.
The first part of the following theorem which is a generalization of the Matrix-Tree Theorem has appeared in \cite{kel} (see also \cite[p.~193]{crsB}). The second part was proved in \cite{dedo} (see also \cite{crs}).
\begin{thm}\label{coef} The coefficients of $p_L(x)$ and $p_Q(x)$ are determined as follows.
\begin{itemize}
  \item[\rm(i)] $\ell_j=(-1)^j\sum_{F_j}W(F_j)$, for $j=1,\ldots,n-1$,
where the summation runs over all acyclic subgraphs $F_j$ of $G$ with $j$ edges.
  \item[\rm(ii)] $q_j=(-1)^j\sum_{H_j}W(H_j)$,  for $j=1,\ldots,n$,
where the summation runs over all TU-subgraphs $H_j$ of $G$ with $j$ edges.
\end{itemize}
\end{thm}

We close this section by stating the following well known lemma for later use.
\begin{lem}\label{princ} Any symmetric matrix of rank $r$ (over any field) has a
principal $r\times r$ submatrix of full rank.
\end{lem}

\section{Binary rank of line graphs}

In this section we give a simple proof of Doob's result. For a matrix $M$, we use the notation $\rk_2(M)$ to denote the binary rank (the rank over the two-element field) of $M$.

\begin{thm}\label{doob} {\rm(Doob \cite{doob})} Let $G$ be a connected graph of order $n$ and $\A=A(\li(G))$. Then
${\rm rank}_2(\A)$ is equal to $n-1$  if $n$ is odd, and $n-2$ if $n$ is even.
\end{thm}

\begin{proof}{
If $S$ is the edge set of a spanning tree and $R$ is any set of $n-1$ vertices of $G$, then by Lemma~\ref{invertX},
$\det(X(R,S))=\pm1$. Hence, $\rk_2(X)\ge n-1$. In fact we have equality since the rows of $X$ sum up to the all 2 vector.
From (\ref{A+2I}), it follows that $\rk_2(\A)\le\rk_2(X)=n-1$.
Let $S\subseteq E(G)$ with $|S|=n-1$. By the Binet--Cauchy Theorem and Lemma~\ref{invertX},
\begin{align*}
\det(\A(S,S))&\equiv\det((X^\p X)(S,S))\pmod{2}\\
&=\sum_{R\subseteq V(G),\,|R|=n-1}\det(X(R,S))^2
=\left\{\begin{array}{ll} n & \hbox{if $\langle S\rangle$ is a tree,} \\0 & \hbox{otherwise.}\end{array}\right.
\end{align*}
This shows that $\A$ has a principal submatrix of order $n-1$ with full binary rank if $n$ is odd and does not if $n$ is even.
This proves the theorem for odd $n$.
Assume that $n$ is even. The above argument together with Lemma~\ref{princ} show that $\A$ has no principal submatrices of order $n-1$ with full binary rank.
Thus  $\rk_2(\A)\le n-2$.
Let $T$ be a subtree of $G$ with $n-2$ edges. Then the adjacency matrix ${\cal B}$ of $\li(T)$ is a principal submatrix of $\A$ and further more  ${\cal B}$ has full binary rank by applying the same argument as for odd $n$.
This shows that $\rk_2(\A)=n-2$.
}\end{proof}

\section{Even integer eigenvalues of Laplacian and signless Laplacian}

In this section we demonstrate the interconnection between the number of spanning trees of a graph $G$ and the even integer eigenvalues of $L(G)$ and $Q(G)$.
In view of the fact that the matrices $A(\li(G))+2I$ and $Q(G)$ have the same nonzero eigenvalues,  Theorems~\ref{t+1}, \ref{mult2line}, and \ref{oddline} follow respectively from Theorems~\ref{t+1Q}, \ref{nodd}, and \ref{mult2} below.

\begin{thm}\label{t+1Q} Let $G$ be a connected graph having $2^ts$ spanning trees with $s$ odd. Then the multiplicity of any even integer $\la$ as an eigenvalue of $Q(G)$ or $L(G)$ is at most $t+1$.
\end{thm}
\begin{proof}{
It is well known that for a given integral
matrix $A$ of rank $r$, there exist unimodular matrices (that is, integral matrices with
determinant $\pm1$) $U$ and $V$
such that $$UAV={\rm diag}(s_1, \ldots ,s_r, 0, \ldots, 0)$$ where
$s_1, \ldots ,s_r$ are positive integers with $s_1s_2\cdots s_i=d_i$ where $d_i$ is the greatest common divisor of all minors
of $A$ of order $i$, $1\leq i\leq r$.
(The matrix ${\rm diag}(s_1, \ldots ,s_r, 0, \ldots, 0)$ is called the {\em Smith form} of $A$.)

Let $S={\rm diag}(s_1, \ldots ,s_{n-1},0)$ be the Smith form of $L$. Note that
${\rm rank}_2(L)={\rm rank}_2(S)$.
By the Matrix-Tree Theorem, $\tau(G)=d_{n-1}=s_1s_2\cdots s_{n-1}$. It follows that at most $t$ of the $s_i$ are even. Therefore, ${\rm rank}_2(L)\ge n-t-1$ and so ${\rm rank}_2(Q)\ge n-t-1$.
By Lemma~\ref{princ}, both $Q$ (and also $L$) has a  principal
submatrix $B$ of order $k\ge n-t-1$ with full binary rank. By interlacing, if an even integer $\la$ is an eigenvalue of $Q$ (or $L$) with multiplicity at least $t+2$, then any principal submatrix of $Q$ (or $L$) of order $k\ge n-t-1$ has $\la$ as an eigenvalue.
So $\la$ is an eigenvalue of $B$. This implies that $\det(B)/\la$ is a rational algebraic integer and thus an integer. Hence $\det(B)$ is even, a contradiction. This completes the proof.
}\end{proof}

\begin{rem} The bound `$t+1$' of Theorem~\ref{t+1Q} on the multiplicity of even eigenvalues of $Q$ and $L$ is best possible. For, if we let $G$ to be the complete graph of order $n\equiv2\pmod4$, then by Cayley's Formula, $\tau(G)=n^{n-2}=2^{n-2}s$ for some odd $s$, and $Q(G)$ has the even integer $n-2$ as an eigenvalue of multiplicity $n-1$.
Also, $L(G)$ has  $n$ as an eigenvalue of multiplicity $n-1$.
\end{rem}

Suppose that $G$ is a connected graph with $n$ vertices,  $e(G)$ edges and $\A=A(\li(G))$.
By the same argument as the proof of Theorem~\ref{t+1Q}, we see that the multiplicity of any even integer  eigenvalue $\la$ of $\A$ is at most $e(G)-\rk_2(\A)$.
Therefore, in view of Theorem~\ref{doob}, the multiplicity of $\la$ is at most $e(G)-2\lceil n/2\rceil+2$. Combination with Theorem~\ref{t+1} yields the following result.

\begin{thm} Let $G$ be a connected graph with $n$ vertices, $e(G)$ edges, and $2^ts$ spanning trees with $s$ odd. Then the multiplicity of any even integer $\la\ne-2$ as an eigenvalue of $A(\li(G))$ is at most $\min\{t+1,e(G)-2\lceil n/2\rceil+2\}$.
\end{thm}

In the rest of the paper, we shall need a variation of Theorem~\ref{coef} on the coefficients of the characteristic polynomials of
principal submatrices of order $n-1$ of $L(G)$ and $Q(G)$. For simplicity, we denote by $L_1=L_1(G)$ and $Q_1=Q_1(G)$ the matrices obtained from
 $L(G)$ and $Q(G)$ by removing the first row and the first column, respectively.
 Note that  $L_1(G)$ and $Q_1(G)$ are not the same as $L(G-v_1)$ and $Q(G-v_1)$ where $v_1$ is the vertex corresponding to the first rows of $L(G)$ and $Q(G)$.

A notion of `restricted weight' with respect to $v_1$ is useful to
describe the coefficients of $p_{L_1}(x)$ and $p_{Q_1}(x)$.
Let $U$ be a unicyclic subgraph of $G$ with odd cycle and $T$ be a tree subgraph of $G$. We define
$$W_1(U)=\left\{\begin{array}{ll} 0 & \hbox{if $U$ contains $v_1$,} \\4 & \hbox{otherwise,}\end{array}\right.~~\hbox{and}~~~
W_1(T)=\left\{\begin{array}{ll} 1 & \hbox{if $T$ contains $v_1$,} \\1+e(T) & \hbox{otherwise.}\end{array}\right.$$
We extend the domain of $W_1$ to all TU-subgraphs $H$ of $G$ by defining $W_1(H)$ to be the product of the $W_1$'s of the connected components of $H$.

\begin{lem}\label{coefnn} Let $p_{L_1}(x)=x^{n-1}+\ell'_1x^{n-2}+\cdots+\ell'_{n-1}$ and $p_{Q_1}(x)=x^{n-1}+q'_1x^{n-2}+\cdots+q'_{n-1}$
be the characteristic polynomials of $L_1$ and $Q_1$, respectively.
Then their coefficients are determined as follows.
\begin{itemize}
  \item[\rm(i)] $\ell'_j=(-1)^j\sum_{F_j}W_1(F_j)$, for $j=1,\ldots,n-1$,
where the summation runs over all spanning forests $F_j$ of $G$ with $j$ edges.
  \item[\rm(ii)] $q'_j=(-1)^j\sum_{H_j}W_1(H_j)$,  for $j=1,\ldots,n-1$,
where the summation runs over all TU-subgraphs $H_j$ of $G$ with $j$ edges.
\end{itemize}
\end{lem}
\begin{proof}{ Let $E=E(G)$  and $V_1=V(G)\setminus\{v_1\}$.

 (i)
 For $j=1,\ldots,n-1$, we have $$\ell'_j =(-1)^j \sum_{R\subseteq V_1,|R|=j}\det(L(R,R)).$$
 From the Binet--Cauchy Theorem it follows that
$$\det(L(R,R)) = \sum_{S\subseteq E,|S|=j}\det(D(R, S))^2.$$
Thus,
\begin{equation}\label{l'j}
\ell'_j =(-1)^j \sum\det(D(R, S))^2,
\end{equation}
 where the summation is over $R\subseteq V_1$, $S\subseteq E$ with $|R|=|S|=j$.
Now $\det(D(R,S))^2$ is either $0$ or $1$ by Lemma~~\ref{invertD}. Further, it takes the value 1 if and only if the three conditions of Lemma~\ref{invertD} hold.
Hence, if $\det(D(R,S))^2=1$, then $\langle S\rangle$ must be a union of some trees $T_1,\ldots,T_r$.
For such a subset $S$ of vertex labels, the contribution of $\langle S\rangle$ in (\ref{l'j}) is the number of $R\subseteq V_1$, $|R|=j$ such that $R$ is obtained by omitting one vertex from each $V(T_i)$.
Assume that $v_1$ is contained in $T_1$. Since $v_1\not\in V_1$, besides this vertex, we have no more options to omit any vertex of $T_1$.
For the other components $T_i$, $i=2,\ldots,r$, we have $1+e(T_i)$   ways of omitting one vertex.
It follows that the contribution of $\langle S\rangle$ in (\ref{l'j}) is $(1+e(T_2))\cdots(1+e(T_r))$ which is equal to $W_1(\langle S\rangle)$.

(ii) The proof is similar to that of part (i). The only points different from part (i) are that here we use Lemma~\ref{invertX} instead of Lemma~\ref{invertD} and that  if $\langle S\rangle$ is a TU-subgraph and some unicyclic component of $\langle S\rangle$ contains $v_1$, then for any $R\subseteq V_1$, $\det(X(R,S))=0$. Hence any TU-subgraph $\langle S\rangle$ with nonzero contribution in $q'_j$ must have all of its unicyclic components included in $V_1$.
}\end{proof}

\begin{thm}\label{nodd} Suppose that $G$ is a connected graph with an odd order and $\tau(G)$ is not divisible by $4$.
Then
\begin{itemize}
  \item[\rm(i)] $L(G)$ has no nonzero even eigenvalues;
  \item[\rm(ii)] $Q(G)$ has no integer eigenvalue $\la\equiv2\pmod4$;
  \item[\rm(iii)] $Q(G)$ has at most one eigenvalue $\la\equiv0\pmod4$ and such an  eigenvalue is  simple.
\end{itemize}
\end{thm}
\begin{proof}{Let $G$ be of order $n$.

(i) We claim that the coefficient $\ell_{n-2}$ of the characteristic polynomial $p_L(x)=x^n+\ell_1x^{n-1}+\cdots+\ell_{n-1}x$ of $L(G)$ is even.
By Theorem~\ref{coef}, we have $\ell_{n-2}=(-1)^{n-2}\sum_{F_{n-2}}W(F_{n-2})$ where the summation runs over all spanning forests $F_{n-2}$ of $G$ with $n-2$ edges.
Any $F_{n-2}$ is necessarily a union of two trees $T_1$ and $T_2$ with $e(T_1)+e(T_2)=n-2$.  As $n$ is odd, $$W(F_{n-2})=(1+e(T_1))(1+e(T_2))$$ is even.
This shows that $\ell_{n-2}$ is even.
Let $k$ be an even integer where $k=2^ts$ with $s$ odd and $t\ge1$.
Then, all the terms of $p_L(k)$ are divisible by $2^{t+2}$ except the last term, namely $\ell_{n-1}k=(-1)^{n-1}n\tau(G)k$ which is congruent to $2^t\tau(G)\pmod{2^{t+2}}$.
Therefore, $p_L(k)\equiv2^{t+1}\pmod{2^{t+2}}$ if $\tau(G)\equiv2\pmod4$ and $p_L(k)\equiv2^t\pmod{2^{t+2}}$ if $\tau(G)$ is odd. This proves (i).

(ii) From Theorem~\ref{coef} it follows that for some integers $s_1,\ldots,s_n$, we have $p_j=\ell_j+4s_j$, $j=1,\ldots,n-1$, and $p_n=4s_n$.
This implies that $p_Q(x)=p_L(x)+4f(x)$ where $f(x)$ is a polynomial with integer coefficients. First assume that $\tau(G)$ is odd. Hence $\ell_{n-1}=(-1)^{n-1}n\tau(G)$ is an odd integer. It follows that if $k\equiv2\pmod4$, then $p_Q(k)\equiv2\pmod4$, and we are done.
Next assume that $\tau(G)\equiv2\pmod4$.
We claim that $s_n$, the constant term of $f(x)$, is even.
Let $\cal U$ be the set of all spanning unicyclic subgraphs of $G$.
By Theorem~\ref{coef}\,(ii), $s_n$ is the number of $U\in\cal U$ such that the cycle of $U$ has an odd length.
If $s_n=0$, we are done.  So assume that $s_n\ge1$. % and thus $G$ is not bipartite.
Let $\cal F$ be the set of all pairs $(T,U)$ such that $U\in\cal U$ and $T$ is a spanning tree of $U$.
For any fixed $U$, the number of pairs $(T,U)\in\cal F$ is equal to the length of the cycle of $U$.
Therefore, $|\cal F|$ is congruent to $s_n$ mod 2.
On the other hand, for any spanning tree $T$ of $G$, there are exactly $e(G)-n+1$ unicyclic graphs $U\in\cal U$ containing $T$.
Therefore, $|{\cal F}|=\tau(G)(e(G)-n+1)\equiv0\pmod2$.
It follows that $s_n$ is even which in turn implies that $f(k)$ is even.
On the other hand, from the proof of part (i) we see that $p_L(k)\equiv n\tau(G)k\equiv4\pmod8$.
Therefore, $p_Q(k)=p_L(k)+4f(k)\equiv4\pmod8$.

(iii)
Suppose that $Q(G)$
has an even integer eigenvalue $\la$.
By part (ii), $\la\equiv0\pmod4$.
If the multiplicity of $\la$ is more than $1$, then $\la$ is an eigenvalue of $Q_1$.
Take $\theta$ as $\theta\la=\det(Q_1)$.
Then $\theta$ is a rational algebraic integer, and so it is an integer.
It follows that $\la$ divides $\det(Q_1)$ which means $\det(Q_1)\equiv0\pmod4$.
On the other hand, from Lemma~\ref{coefnn} it follows  $p_{Q_1}(x)=p_{L_1}(x)+4f_1(x)$ for some integer polynomial $f_1(x)$.
This implies that $\det(Q_1)\equiv\det(L_1)\equiv\tau(G)\pmod4$
 which is a contradiction.
}\end{proof}

\begin{rem} Note that if $n$ is odd and we let $G$ to be the complete graph of order $n$, then by Cayley's Formula, $\tau(G)=n^{n-2}$ is odd.
Also $2n-2$ is the largest eigenvalue of $Q(G)$.
Hence $Q(G)$ has an eigenvalue divisible by $4$.
 This shows that Theorem~\ref{nodd}\,(iii) cannot be improved.
\end{rem}

\begin{thm}\label{mult2} Let $G$ be a connected graph. If $L(G)$ or $Q(G)$ has an even integer eigenvalue of multiplicity at least $2$, then $\tau(G)$ is divisible by $4$.
\end{thm}
\begin{proof}{Let $G$ be of order $n$. If $n$ is odd we are done by Theorem~\ref{nodd}. So we may assume that $n$ is even.
Let $v_1\in V(G)$ correspond to the first row of $L(G)$ and $Q(G)$.

First, let $\la$ be an even integer eigenvalue of $L(G)$ with multiplicity at least $2$.
By the interlacing property of eigenvalues of Hermitian matrices, $\la$ is also an eigenvalue of $L_1(G)$. Let $p_{L_1}(x)=x^{n-1}+\ell'_1x^{n-2}+\cdots+\ell'_{n-1}$  be the characteristic polynomials of $L_1$.
By the Matrix-Tree Theorem,
$\ell'_{n-1}=(-1)^{n-1}\tau(G)$.  We show that $\ell'_{n-2}$ is even.
Any spanning forest of $G$ with $n-2$ edges is a union of two trees $T_1$ and $T_2$ where we may assume that $T_1$ contains the vertex $v_1$, and hence by Lemma~\ref{coefnn}, $\ell'_{n-2}=(-1)^{n-2}\sum_{T_1\cup T_2}(1+e(T_2))$.
We note that
\begin{equation}\label{eT1eT2}
(1+e(T_1))(1+e(T_2))\equiv(1+e(T_2))\pmod2,
\end{equation}
 for if $e(T_2)$ is even, then $e(T_1)$ is also even as $e(T_1)+e(T_2)=n-2$ so both sides of \eqref{eT1eT2} are odd, and if $e(T_2)$ is odd both sides of  \eqref{eT1eT2} are even.
This implies that
$$\ell'_{n-2}=(-1)^{n-2}\sum_{T_1\cup T_2}(1+e(T_2))\equiv(-1)^{n-2}\sum_{T_1\cup T_2}(1+e(T_1))(1+e(T_2))=\ell_{n-2} \pmod2.$$
Note that $p_L(x)=(x-\la)^2g(x)$ for some integer polynomial $g(x)$.
If  $ax^2+bx$ are the last two terms of $g(x)$, then  $\ell_{n-2}=\la^2a-2\la b$. It follows that $\ell_{n-2}$ and so $\ell'_{n-2}$ is even.
Therefore, $$0=p_{L_1}(\la)\equiv\ell'_{n-1}=(-1)^{n-1}\tau(G)\pmod4.$$

Now let $\la$ be an even integer eigenvalue of $Q(G)$ with multiplicity at least $2$.
So $\la$ is also an eigenvalue of $Q_1(G)$.
From Theorem~\ref{coefnn} it follows  $p_{Q_1}(x)=p_{L_1}(x)+4f_1(x)$ for some integer polynomial $f_1(x)$.
Therefore, $0=p_{Q_1}(\la)\equiv p_{L_1}(\la)\equiv\tau(G)\pmod4.$
}\end{proof}

As a very special case of Theorem~\ref{mult2} we deduce the following result which was conjectured in \cite{zsyb}.
\begin{cor} Suppose that $G$ is a unicyclic graph and the nullity of $A(\li(G))$ is equal $2$. Then the length of the unique cycle of $G$ is divisible by $4$.
\end{cor}

More general assertions than Theorems~\ref{nodd} and \ref{mult2} hold for the Laplacian matrix. These are given below. We omit the proof which is essentially the same as the proofs of
Theorems~\ref{nodd} and \ref{mult2}.
\begin{thm}
Let $G$ be a connected graph of order $n$.
\begin{itemize}
  \item[\rm(i)] If $n$ is odd and $\tau(G)=2^ts$ with $s$ odd, then $L(G)$ has no nonzero eigenvalue $\la$ such that $2^{\max(1,t)}$ divides $\la$.
 \item[\rm(ii)] If $L(G)$ has an integer eigenvalue $\la=2^ts$ with $t\ge1$, $s$ odd and with multiplicity at least $2$,
then $2^{t+1}$ divides $\tau(G)$.
\end{itemize}
\end{thm}

\section*{Acknowledgments}
  I would like to thank Ali Mohammadian for drawing my attention to the conjecture given in \cite{zsyb} which motivated me to establish Theorem~\ref{mult2}.  I also thank anonymous referees for several helpful comments and corrections.
   The research of the  author was partially supported by a grant from IPM (No. 92050114).

\end{document}